\documentclass{article}

\usepackage{amsmath,amssymb,amsfonts,amsthm}

\usepackage{cite}

\theoremstyle{definition}
\newtheorem{definition}{Definition}[section]
\theoremstyle{theorem}
\newtheorem{theorem}{Theorem}[section]

\newtheorem{lemma}[theorem]{Lemma}
\newtheorem{remark}[theorem]{Remark}

\begin{document}

\title{Fractional Calculus of Variations\\ for Double Integrals\thanks{Submitted 07-Sept-2010; 
revised 25-Nov-2010; accepted 07-Feb-2011; for publication in 
\emph{Balkan Journal of Geometers and Its Applications} (BJGA).}}

\author{Tatiana Odzijewicz\\ 
{\tt tatianao@ua.pt}
\and 
Delfim F. M. Torres\\
{\tt delfim@ua.pt}}

\date{Department of Mathematics\\
University of Aveiro\\
3810-193 Aveiro, Portugal}

\maketitle


\vspace*{-0.6cm}

\begin{center}
\emph{Dedicated to Professor Constantin Udri\c ste
on the occasion of his 70th birthday}
\end{center}


\begin{abstract}
We consider fractional isoperimetric problems
of calculus of variations with double integrals
via the recent modified Riemann--Liouville approach.
A necessary optimality condition of Euler--Lagrange type,
in the form of a multitime fractional PDE, is proved,
as well as a sufficient condition
and fractional natural boundary conditions.

\medskip

\noindent {{\bf M.S.C. 2010}:\ } 49K21, 35R11.

\medskip


\end{abstract}


\section{Introduction}

The calculus of variations was born in 1697 with the
solution to the brachistochrone problem
(see, \textrm{e.g.}, \cite{Brunt}).
It is a very active research area in the XXI century
(see, \textrm{e.g.}, \cite{MyID152,MyID:175,MyID:141,MyID:171,MyID130}).
Motivated by the study of several natural phenomena in such areas as
aerodynamics, economics, medicine, environmental engineering, and biology,
there has been a recent increase of interest in the study of
problems of the calculus of variations and optimal control
where the cost is a multiple integral functional with several
independent time variables. The reader interested in the area
of multitime calculus of variations and multitime optimal control
is referred to \cite{Marutani,P:Udriste:mt,Udriste:JOTA,Udriste:mtMP,%
MR2539748,Udriste:mt,Udriste:mt2,Udriste:Pop:Pop,MR2110147:mt,Udriste:Tevy}
and references therein.

Fractional calculus, \textrm{i.e.},
the calculus of non-integer order derivatives,
has its origin also in the 1600s.
During three centuries the theory of fractional derivatives
of real or complex order developed as a pure theoretical
field of mathematics, useful only for mathematicians.
In the last few decades, however, fractional differentiation
proved very useful in various fields of applied sciences and engineering:
physics (classic and quantum mechanics, thermodynamics, etc.),
chemistry, biology, economics, engineering,
signal and image processing, and control theory
\cite{MyID179,Hilfer,Kilbas,MyID163,MyID181,podlubny}.

The calculus of variations and the fractional calculus
are connected since the XIX century. Indeed, in 1823
Niels Henrik Abel applied the fractional calculus
in the solution of an integral equation
that arises in the formulation of the tautochrone problem.
This problem, sometimes also called the isochrone problem,
is that of finding the shape of a frictionless wire lying in a vertical
plane such that the time of a bead placed on the wire slides
to the lowest point of the wire in the same time regardless
of where the bead is placed. It turns out that the cycloid
is the isochrone as well as the brachistochrone curve,
solving simultaneously the brachistochrone problem
of the calculus of variations
and Abel's fractional problem \cite{Abel}.
It is however in the XX century that both areas are joined
in a unique research field: the fractional calculus of variations.

The Fractional Calculus of Variations (FCV) was born in 1996-97
with the proof, by Riewe, of the Euler-Lagrange fractional
differential equations \cite{Riewe:1996,Riewe:1997}.
Nowadays, FCV is subject of strong current research -- see, \textrm{e.g.},
\cite{OmPrakashAgrawal,FrTor1,Baleanu,FrIso,FrDI,FrMult,NatFr}.
The first works on FCV were developed using fractional
derivatives in the sense of Riemann--Liouville \cite{OmPrakashAgrawal}.
Later, problems of FCV with Grunwald--Letnikow, Caputo, Riesz and Jumarie fractional operators,
among others, were considered \cite{FrMult,NatFr,NabTor,FrTor2}.
The literature on FCV is now vast. However, most results refer to the single time case.
Results for multitime FCV are scarce, and reduce to those in \cite{FrMult,NabTor:JMP,Udriste:fa}.
Here we develop further the theory of multitime fractional calculus of variations,
by considering fractional isoperimetric problems with two independent time variables.
Previous results on fractional isoperimetric problems are for the single time case only
\cite{FrIso,FrDI}. In our paper we study isoperimetric problems for variational functionals
with double integrals involving fractional partial derivatives.

The paper is organized as follows. In Section~2 we recall some basic definitions
of multidimensional fractional calculus. Our results are stated and proved in Section~3.
The main results of the paper include natural boundary conditions (Theorem~\ref{thm:NatBound})
and a necessary optimality condition (Theorem~\ref{thm:EL}) that becomes sufficient
under appropriate convexity assumptions (Theorem~\ref{thm:Suff}).


\section{Preliminaries}
\label{sec:fdRL}

In this section we fix notations by collecting the
definitions of fractional derivatives and integrals
in the modified Riemann--Liouville sense. For more
information on the subject we refer the reader to
\cite{FrMult,JuFr,JuMod,JuTab,withBasiaRachid}.

\begin{definition}[The Jumarie fractional derivative \cite{JuTab}]
Let $f$ be a continuous function in the interval
$[a,b]$ and $\alpha\in (0,1)$. The operator defined by
\begin{equation}
\label{eq:fracJum}
f^{(\alpha)}(x)=\frac{1}{\Gamma(1-\alpha)}\frac{d}{dx}
\int\limits_a^x(x-t)^{-\alpha}(f(t)-f(a))dt
\end{equation}
is called the Jumarie fractional derivative of order $\alpha$.
\end{definition}
Let us consider continuous functions
$f=f(x_1,\ldots,x_n)$ defined on
$$
R=\prod\limits_{i=1}^n[a_i,b_i]\subset\mathbb{R}^n .
$$
\begin{definition}[The fractional volume integral \cite{FrMult}]
For $\alpha\in (0,1)$ the fractional volume integral
of $f$ over the whole domain $R$ is given by
\begin{equation*}
I_R^\alpha f = \alpha^n\int\limits_{a_1}^{b_1} \ldots
\int\limits_{a_n}^{b_n} f(t_1,\ldots,t_n)(b_1-t_1)^{\alpha-1}
\ldots (b_n-t_n)^{\alpha-1} dt_n \ldots dt_1.
\end{equation*}
\end{definition}

\begin{definition}[Fractional partial derivatives \cite{FrMult}]
\label{dfn:FracPart}
Let $x_i\in [a_i,b_i]$, $i=1,\ldots,n$, and $\alpha\in (0,1)$.
The operator ${_{a_i}D_{x_i}^\alpha}[i]$ defined by
\begin{multline*}
_{a_i}D_{x_i}^\alpha[i]f(x_1,\ldots,x_n)
=\frac{1}{\Gamma(1-\alpha)}\frac{\partial}{\partial x_i}
\int\limits_{a_i}^{x_i}(x_i-t)^{-\alpha}\Bigr[f(x_1,
\ldots,x_{i-1},t,x_{i+1},\ldots,x_n)\\
-f(x_1,\ldots,x_{i-1},a_i,x_{i+1},\ldots,x_n)\Bigr] dt
\end{multline*}
is called the $i$th fractional partial derivative
of order $\alpha$, $i=1,\ldots,n$.
\end{definition}

\begin{remark}
The Jumarie fractional derivative \cite{JuMod,JuTab}
given by \eqref{eq:fracJum} can be obtained by putting
$n=1$ in Definition~\ref{dfn:FracPart}:
\begin{equation*}
_{a}D_{x}^\alpha[1]f(x)=\frac{1}{\Gamma(1-\alpha)}\frac{d}{dx}
\int\limits_a^x(x-t)^{-\alpha}(f(t)-f(a))dt=f^{(\alpha)}(x).
\end{equation*}
\end{remark}

\begin{definition}[The fractional line integral \cite{FrMult}]
Let $R=[a,b]\times[c,d]$. The fractional line integral
on $\partial R$ is defined by
\begin{equation*}
I_{\partial R}^\alpha f
=I_{\partial R}^\alpha[1]f+I_{\partial R}^\alpha[2]f,
\end{equation*}
where
\begin{equation*}
I_{\partial R}^\alpha[1]f=\alpha
\int\limits_a^b\left[f(t,c)-f(t,d)\right](b-t)^{\alpha-1}dt
\end{equation*}
and
\begin{equation*}
I_{\partial R}^\alpha[2]f=\alpha
\int\limits_c^d\left[f(b,t)-f(a,t)\right](d-t)^{\alpha-1}dt.
\end{equation*}
\end{definition}


\section{Main Results}
\label{sec-FrCVDI}

Let us consider functions $u=u(x,y)$.
We assume that the domain of functions $u$ contain
the rectangle $R=[a,b]\times[c,d]$ and are continuous on $R$.
Moreover, functions $u$ under our consideration
are such that the fractional partial derivatives
$_{a}D_{x}^\alpha[1]u$ and $_{c}D_{y}^\alpha[2]u$ are continuous on $R$,
$\alpha\in(0,1)$. We investigate the following
fractional problem of the calculus of variations:
to minimize a given functional
\begin{equation}
\label{eq:Funct}
J[u(\cdot,\cdot)]=\alpha^2\int\limits_a^b\int\limits_c^d
f\left(x,y,u,_{a}D_{x}^\alpha[1]u,_{c}D_{y}^\alpha[2]u\right)
(b-x)^{\alpha-1}(d-y)^{\alpha-1}dy dx
\end{equation}
when subject to an isoperimetric constraint
\begin{equation}
\label{eq:IFunct}
\alpha^2\int\limits_a^b\int\limits_c^d
g\left(x,y,u,_{a}D_{x}^\alpha[1]u,_{c}D_{y}^\alpha[2]u\right)
(b-x)^{\alpha-1}(d-y)^{\alpha-1}dydx=K
\end{equation}
and a boundary condition
\begin{equation}
\label{eq:bound}
\left.u(x,y)\right|_{\partial R}=\psi(x,y).
\end{equation}
We are assuming that $\psi$ is some given function,
$K$ is a constant, and $f$ and $g$ are at least of class of $C^1$.
Moreover, we assume that $\partial_{4} f$ and $\partial_{4} g$
have continuous fractional partial derivatives $_{a}D_{x}^\alpha[1]$;
and $\partial_{5} f$ and $\partial_{5} g$ have continuous fractional
partial derivatives $_{c}D_{y}^\alpha[2]$. Along the work, we denote by
$\partial_{i}f$ and $\partial_{i}g$ the standard partial derivatives
of $f$ and $g$ with respect to their $i$-th argument, $i = 1,\ldots,5$.

\begin{definition}
A continuous function $u=u(x,y)$ that satisfies the given
isoperimetric constraint \eqref{eq:IFunct}
and boundary condition \eqref{eq:bound},
is said to be admissible for problem \eqref{eq:Funct}-\eqref{eq:bound}.
\end{definition}

\begin{remark}
Contrary to the classical setting
of the calculus of variations, where admissible functions
are necessarily differentiable, here we are considering our
variational problem \eqref{eq:Funct}-\eqref{eq:bound}
on the set of continuous curves $u$
(without assuming differentiability of $u$). Indeed,
the modified Riemann--Liouville derivatives have
the advantage of both the standard Riemann--Liouville
and Caputo fractional derivatives: they are defined for arbitrarily
continuous (not necessarily differentiable) functions, like the
standard Riemann--Liouville ones, and the fractional derivative
of a constant is equal to zero, as it happens with the Caputo derivatives.
\end{remark}

\begin{definition}[Local minimizer to \eqref{eq:Funct}-\eqref{eq:bound}]
An admissible function $u = u(x,y)$ is said to be a local minimizer
to problem \eqref{eq:Funct}-\eqref{eq:bound}
if there exists some $\gamma>0$ such that for all admissible functions
$\hat{u}$ with $\left\|\hat{u}-u\right\|_{1,\infty} < \gamma$ one has
$J[\hat{u}]-J[u]\geq 0$, where
\begin{equation*}
\|u\|_{1,\infty}:=\max_{(x,y) \in R}\left|u(x,y)\right|
+\max_{(x,y) \in R}\left|_{a}D_{x}^\alpha[1]u(x,y)\right|
+\max_{(x,y) \in R}\left|_{c}D_{y}^\alpha[2]u(x,y)\right|.
\end{equation*}
\end{definition}

We make use of the following result proved in \cite{FrMult}:

\begin{lemma}[Green's fractional formula \cite{FrMult}]
\label{lemma:Green}
Let $h,k$, and $\eta$ be continuous
functions whose domains contain $R$. Then,
\begin{equation*}
\begin{split}
&\int\limits_a^b\int\limits_c^d\left[h(x,y)_{a}D_{x}^\alpha[1]\eta(x,y)
-k(x,y)_{c}D_{y}^\alpha[2]\eta(x,y)\right](b-x)^{\alpha-1}(d-y)^{\alpha-1}dydx\\
&=-\int\limits_a^b\int\limits_c^d\left[_{a}D_{x}^\alpha[1]h(x,y)
-_{c}D_{y}^\alpha[2]k(x,y)\right]\eta(x,y)(b-x)^{\alpha-1}(d-y)^{\alpha-1}dydx\\
&\quad +\alpha !\left[I_{\partial R}^\alpha[1](h\eta)+I_{\partial R}^\alpha[2](k\eta)\right].
\end{split}
\end{equation*}
\end{lemma}

\begin{remark}
If $\eta\equiv 0$ on $\partial R$ in Lemma~\ref{lemma:Green}, then
\begin{multline}
\label{eq:MoreParticularGreen}
\int\limits_a^b\int\limits_c^d\left[h(x,y)_{a}D_{x}^\alpha[1]\eta(x,y)
-k(x,y)_{c}D_{y}^\alpha[2]\eta(x,y)\right](b-x)^{\alpha-1}(d-y)^{\alpha-1}dydx\\
=-\int\limits_a^b\int\limits_c^d\left[_{a}D_{x}^\alpha[1]h(x,y)
-_{c}D_{y}^\alpha[2]k(x,y)\right]\eta(x,y)(b-x)^{\alpha-1}(d-y)^{\alpha-1}dydx.
\end{multline}
\end{remark}


\subsection{Necessary Optimality Condition}
\label{sub:sec:FrIP}

The next theorem gives a necessary optimality condition for $u$
to be a solution of the fractional isoperimetric problem defined
by \eqref{eq:Funct}-\eqref{eq:bound}.

\begin{theorem}[Euler--Lagrange fractional optimality
condition to \eqref{eq:Funct}-\eqref{eq:bound}]
\label{thm:EL}
If $u$ is a local minimizer
to problem \eqref{eq:Funct}-\eqref{eq:bound},
then there exists a nonzero pair of constants
$(\lambda_0,\lambda)$ such that $u$ satisfies
the fractional PDE
\begin{equation}
\label{eq:EL}
\partial _3 H\left\{u\right\}(x,y)
- _{a}D_{x}^\alpha[1]\partial _4 H\left\{u\right\}(x,y)
- _{c}D_{y}^\alpha[2]\partial_5 H\left\{u\right\}(x,y)=0
\end{equation}
for all $(x,y)\in R$, where
$$
H(x,y,u,v,w,\lambda_0,\lambda):=\lambda_0 f(x,y,u,v,w)+\lambda g(x,y,u,v,w)
$$
and, for simplicity of notation, we use the operator $\left\{\cdot\right\}$ defined by
\begin{equation*}
\left\{u\right\}(x,y)
:=\left(x,y,u(x,y),_{a}D_{x}^\alpha[1]u(x,y),_{c}D_{y}^\alpha[2]u(x,y),\lambda_0,\lambda\right).
\end{equation*}
\end{theorem}

\begin{proof}
Let us define the function
\begin{equation}
\label{eq:curves}
\hat{u}_\varepsilon(x,y)=u(x,y)+\varepsilon\eta(x,y),
\end{equation}
where $\eta$ is such that $\eta\in C^1(R)$,
\begin{equation*}
\left.\eta(x,y)\right|_{\partial R}=0,
\end{equation*}
and $\varepsilon\in\,\mathbb{R}$. If $\varepsilon$ take
values sufficiently close to zero, then
\eqref{eq:curves} is included into the first
order neighborhood of $u$, \textrm{i.e.}, there
exists $\delta>0$ such that $\hat{u}_\varepsilon \in U_1(u,\delta)$, where
\begin{equation*}
U_1(u,\delta)
:= \left\{\hat{u}(x,y) : \left\|u - \hat{u}\right\|_{1,\infty} <\delta \right\}.
\end{equation*}
On the other hand,
\begin{equation*}
\hat{u}_0(x,y)=u, \frac{\partial \hat{u}_\varepsilon(x,y)}{\partial \varepsilon}=\eta,
\frac{\partial _{a}D_{x}^\alpha[1] \hat{u}_\varepsilon(x,y)}{\partial \varepsilon}
=_{a}D_{x}^\alpha[1]\eta, \frac{\partial _{c}D_{y}^\alpha[2]\hat{u}_\varepsilon(x,y)}{\partial\varepsilon}
=_{c}D_{y}^\alpha[2]\eta.
\end{equation*}
Let
\begin{equation*}
F(\varepsilon)=\alpha^2\int\limits_a^b\int\limits_c^d
f(x,y,\hat{u}_\varepsilon(x,y),_{a}D_{x}^\alpha[1]\hat{u}_\varepsilon(x,y),
_{c}D_{y}^\alpha[2]\hat{u}_\varepsilon(x,y))(b-x)^{\alpha-1}(d-y)^{\alpha-1}dydx,
\end{equation*}
and
\begin{equation*}
G(\varepsilon)=\alpha^2\int\limits_a^b\int\limits_c^d
g(x,y,\hat{u}_\varepsilon(x,y),
_{a}D_{x}^\alpha[1]\hat{u}_\varepsilon(x,y),_{c}D_{y}^\alpha[2]\hat{u}_\varepsilon(x,y))
(b-x)^{\alpha-1}(d-y)^{\alpha-1}dydx.
\end{equation*}
Define the Lagrange function by
\begin{equation*}
L(\varepsilon,\lambda_0,\lambda)
=\lambda_0F(\varepsilon)+\lambda\left(G(\varepsilon)-K\right).
\end{equation*}
Then, by the extended Lagrange multiplier rule (see, \textrm{e.g.}, \cite{Brunt}),
we can choose multipliers $\lambda_0$ and $\lambda$, not both zero, such that
\begin{equation}
\label{eq:LagrRul}
\frac{\partial L(0,\lambda_0,\lambda)}{\partial\varepsilon}
=\left.\lambda_0\frac{\partial F}{\partial\varepsilon}\right|_{\varepsilon=0}
+\lambda\left.\frac{\partial G}{\partial\varepsilon}\right|_{\varepsilon=0}=0.
\end{equation}
The term $\left.\frac{\partial F}{\partial\varepsilon}\right|_{\varepsilon=0}$
is equal to
\begin{equation}
\label{eq:1}
\begin{split}
\alpha^2 &\int\limits_a^b\int\limits_c^d
\left\{\frac{\partial}{\partial\varepsilon}\left[f(x,y,\hat{u}_\varepsilon,_{a}D_{x}^\alpha[1]\hat{u}_\varepsilon,
_{c}D_{y}^\alpha[2]\hat{u}_\varepsilon)(b-x)^{\alpha-1}(d-y)^{\alpha-1}\right]\right\}_{\varepsilon=0}dydx\\
&=\alpha^2\int\limits_a^b\int\limits_c^d \partial_3 f (b-x)^{\alpha-1}(d-y)^{\alpha-1} dydx \\
& \qquad +\alpha^2\int\limits_a^b\int\limits_c^d \left[\partial_4 f _{a}D_{x}^\alpha[1]\eta
+\partial_5 f _{c}D_{y}^\alpha[2]\eta \right](b-x)^{\alpha-1}(d-y)^{\alpha-1} dydx.
\end{split}
\end{equation}
By \eqref{eq:MoreParticularGreen} the last
double integral in \eqref{eq:1} may be transformed as follows:
\begin{equation*}
\begin{split}
& \alpha^2\int\limits_a^b\int\limits_c^d
\left[\partial_4 f _{a}D_{x}^\alpha[1]\eta
+\partial_5 f _{c}D_{y}^\alpha[2]\eta \right](b-x)^{\alpha-1}(d-y)^{\alpha-1} dydx\\
& \quad =-\alpha^2\int\limits_a^b\int\limits_c^d \left[_{a}D_{x}^\alpha[1]\partial_4 f
+ _{c}D_{y}^\alpha[2]\partial_5 f \right]\eta(b-x)^{\alpha-1}(d-y)^{\alpha-1} dydx.
\end{split}
\end{equation*}
Hence,
\begin{equation}
\label{eq:2}
\left.\frac{\partial F}{\partial\varepsilon}\right|_{\varepsilon=0}
=\alpha^2\int\limits_a^b\int\limits_c^d \left[\partial_3 f -_{a}D_{x}^\alpha[1]\partial_4 f
- _{c}D_{y}^\alpha[2]\partial_5 f\right]\eta(b-x)^{\alpha-1}(d-y)^{\alpha-1} dydx.
\end{equation}
Similarly,
\begin{equation}
\label{eq:3}
\left.\frac{\partial G}{\partial\varepsilon}\right|_{\varepsilon=0}
=\alpha^2\int\limits_a^b\int\limits_c^d \left[\partial_3 g
-_{a}D_{x}^\alpha[1]\partial_4 g
- _{c}D_{y}^\alpha[2]\partial_5 g\right]\eta(b-x)^{\alpha-1}(d-y)^{\alpha-1} dydx.
\end{equation}
Substituting \eqref{eq:2} and \eqref{eq:3} into \eqref{eq:LagrRul}, it results that
\begin{equation*}
\label{eq:4}
\begin{split}
&\frac{\partial L(\varepsilon,\lambda_0,\lambda)}{\partial\varepsilon}
=\alpha^2\int\limits_a^b\int\limits_c^d \Bigl[\lambda_0\left(\partial_3 f
-_{a}D_{x}^\alpha[1]\partial_4 f - _{c}D_{y}^\alpha[2]\partial_5 f\right)\\
&\quad +\lambda\left(\partial_3 g -_{a}D_{x}^\alpha[1]\partial_4 g
- _{c}D_{y}^\alpha[2]\partial_5 g\right)\Bigr]\eta(b-x)^{\alpha-1}(d-y)^{\alpha-1} dydx=0.
\end{split}
\end{equation*}
Finally, since $\eta\equiv0$ on $\partial R$, the fundamental
lemma of the calculus of variations (see, \textrm{e.g.}, \cite{Marutani})
implies that
\begin{equation*}
\partial _3 H\left\{u\right\}(x,y)
- _{a}D_{x}^\alpha[1]\partial _4 H\left\{u\right\}(x,y)
- _{c}D_{y}^\alpha[2]\partial_5 H\left\{u\right\}(x,y)=0.
\end{equation*}
\end{proof}


\subsection{Natural Boundary Conditions}
\label{sub:sec:NT}

In this section we consider problem \eqref{eq:Funct}-\eqref{eq:IFunct},
\textrm{i.e.}, we consider the case when the value of function
$u=u(x,y)$ is not preassigned on $\partial R$.

\begin{theorem}[Fractional natural boundary conditions to \eqref{eq:Funct}-\eqref{eq:IFunct}]
\label{thm:NatBound}
If $u$ is a local minimizer to problem \eqref{eq:Funct}-\eqref{eq:IFunct},
then $u$ is a solution of the fractional differential
equation \eqref{eq:EL}. Moreover, it satisfies the following conditions:
\begin{enumerate}
\item $\partial_4 H \left\{u\right\}(a,y)=0$ for all $y\in[c,d]$;
\item $\partial_4 H \left\{u\right\}(b,y)=0$ for all $y\in[c,d]$;
\item $\partial_5 H \left\{u\right\}(x,c)=0$ for all $x\in[a,b]$;
\item $\partial_5 H \left\{u\right\}(x,d)=0$ for all $x\in[a,b]$.
\end{enumerate}
\end{theorem}

\begin{proof}
Since in problem \eqref{eq:Funct}-\eqref{eq:IFunct} no boundary condition is imposed,
we do not require $\eta$ in the proof o Theorem~\ref{thm:EL} to vanish on $\partial R$.
Therefore, following the proof of Theorem~\ref{thm:EL}, we obtain
\begin{multline}
\label{eq:Nat}
\alpha^2\int\limits_a^b\int\limits_c^d
\left(\partial _3 H\left\{u\right\}(x,y)
+_{a}D_{x}^\alpha[1]\partial_4 H\left\{u\right\}(x,y)\right.\\
\left. +_{c}D_{y}^\alpha[2]
\partial_5 H\left\{u\right\}(x,y)\right)\eta(b-x)^{\alpha-1}(d-y)^{\alpha-1} dydx\\
+\alpha !\left[I_{\partial R}^\alpha[1](\partial _4 H\left\{u\right\}(x,y)\eta)
+I_{\partial R}^\alpha[2](\partial _5 H\left\{u\right\}(x,y) \eta)\right]=0,
\end{multline}
where $\eta$ is an arbitrary continuous function. In particular, the above equation holds
for $\eta\equiv 0$ on $\partial R$. If $\left.\eta(x,y)\right|_{\partial R}=0$,
the second member of the sum in \eqref{eq:Nat} vanishes and the fundamental lemma
of the calculus of variations (see, \textrm{e.g.}, \cite{Marutani}) implies \eqref{eq:EL}.
With this result equation \eqref{eq:Nat} takes the form
\begin{multline}
\label{eq:5}
\int\limits_c^d\partial_4 H \left\{u\right\}(b,y)\eta(b,y)(d-y)^{\alpha-1} dy
-\int\limits_c^d\partial_4 H \left\{u\right\}(a,y)\eta(a,y)(d-y)^{\alpha-1} dy\\
-\int\limits_a^b\partial_5 H \left\{u\right\}(x,c)\eta(x,c)(b-x)^{\alpha-1} dx
-\int\limits_a^b\partial_5 H \left\{u\right\}(x,d)\eta(x,d)(b-x)^{\alpha-1} dx = 0.
\end{multline}
Let $S_1=([a,b]\times {c})\cup([a,b]\times {d})\cup({b}\times[c,d])$.
Since $\eta$ is an arbitrary function, we can consider the subclass of functions
for which $\left.\eta(x,y)\right|_{S_1}=0$.
For such $\eta$, equation \eqref{eq:5} reduces to
\begin{equation*}
0=\int\limits_c^d\partial_4 H \left\{u\right\}(a,y)\eta(a,y)(d-y)^{\alpha-1} dy.
\end{equation*}
By the fundamental lemma of calculus of variations, we obtain that
\begin{equation*}
\partial_4 H \left\{u\right\}(a,y)=0
\end{equation*}
for all $y\in[c,d]$.
We prove the other natural boundary conditions in a similar way.
\end{proof}


\subsection{Sufficient Condition}
\label{sub:sec:OC}

We now prove a sufficient condition that ensures existence
of global minimum under appropriate convexity assumptions.

\begin{theorem}
\label{thm:Suff}
Let $H(x,y,u,v,w,\lambda_0,\lambda)
=\lambda_0 f(x,y,u,v,w)+\lambda g(x,y,u,v,w)$
be a convex function of $u$, $v$ and $w$.
If $u(x,y)$ satisfies \eqref{eq:EL}, then for an arbitrary
admissible function $\hat{u}(\cdot,\cdot)$ the following holds:
\begin{equation*}
J[\hat{u}(\cdot,\cdot)]\geq J[u(\cdot,\cdot)],
\end{equation*}
\textrm{i.e.}, $u(\cdot,\cdot)$ minimizes \eqref{eq:Funct}.
\end{theorem}

\begin{proof}
Define the following function:
\begin{equation*}
\mu(x,y):=\hat{u}(x,y)-u(x,y).
\end{equation*}
Obviously,
\begin{equation*}
\left.\mu(x,y)\right|_{\partial R}=0.
\end{equation*}
Since $H \left\{\hat{u}\right\}(x,y)$ is convex and
$_{a}D_{x}^\alpha[1]$, $_{c}D_{y}^\alpha[2]$ are linear operators,
we obtain that
\begin{equation}
\label{eq:convex}
\begin{split}
&H\left\{\hat{u}\right\}(x,y)-H \left\{u\right\}(x,y)\\
&\geq (\hat{u}(x,y)-u(x,y))\partial _3 H\left\{u\right\}(x,y)
+\left(_{a}D_{a}^\alpha[1]\hat{u}(x,y)-_{a}D_{x}^\alpha[1]u(x,y)\right)
\partial _4 H\left\{u\right\}(x,y)\\
&\quad +\left(_{c}D_{y}^\alpha[2]\hat{u}(x,y)
-_{c}D_{y}^\alpha[2]u(x,y)\right)\partial _5 H\left\{u\right\}(x,y)\\
&=(\hat{u}(x,y)-u(x,y))\partial _3 H\left\{u\right\}(x,y)
+_{a}D_{x}^\alpha[1]\left(\hat{u}(x,y)-u(x,y)\right)
\partial _4 H\left\{u\right\}(x,y)\\
&\quad +_{c}D_{y}^\alpha[2]\left(\hat{u}(x,y)-u(x,y)\right)
\partial _5 H\left\{u\right\}(x,y)\\
&=\mu(x,y)\partial _3 H\left\{u\right\}(x,y)
+_{a}D_{x}^\alpha[1]\mu(x,y) \partial _4 H\left\{u\right\}(x,y)
+_{c}D_{y}^\alpha[2]\mu(x,y)\partial _5 H\left\{u\right\}(x,y),
\end{split}
\end{equation}
where the $\lambda_0$ and $\lambda$ that appear in
$\left\{u\right\}(x,y)$
are constants whose existence
is assured by Theorem~\ref{thm:EL}. Therefore,\footnote{From now on we omit,
for brevity, the arguments $(x,y)$.}
\begin{equation*}
\begin{split}
& J[\hat{u}(\cdot,\cdot)]- J[u(\cdot,\cdot)]\\
& =\alpha^2\int\limits_a^b\int\limits_c^d
f(x,y,\hat{u},_{a}D_{x}^\alpha[1]\hat{u},
_{c}D_{y}^\alpha[2]\hat{u})(b-x)^{\alpha-1}(d-y)^{\alpha-1}dydx\\
& \quad -\alpha^2\int\limits_a^b\int\limits_c^d
f(x,y,u,_{a}D_{x}^\alpha[1]u,_{c}D_{y}^\alpha[2]u)(b-x)^{\alpha-1}(d-y)^{\alpha-1}dydx\\
& \quad +\lambda_0\left(\alpha^2\int\limits_a^b\int\limits_c^d
g(x,y,\hat{u},_{a}D_{x}^\alpha[1]\hat{u},
_{c}D_{y}^\alpha[2]\hat{u})(b-x)^{\alpha-1}(d-y)^{\alpha-1}dydx-K\right)\\
& \quad -\lambda_0\left(\alpha^2\int\limits_a^b
\int\limits_c^dg(x,y,\hat{u},_{a}D_{x}^\alpha[1]\hat{u},
_{c}D_{y}^\alpha[2]u)(b-x)^{\alpha-1}(d-y)^{\alpha-1}dydx-K\right)\\
&  =\alpha^2\int\limits_a^b\int\limits_c^d\left(H\left\{\hat{u}\right\}
-H\left\{u\right\}\right)(b-x)^{\alpha-1}(d-y)^{\alpha-1}dydx .
\end{split}
\end{equation*}
Using \eqref{eq:convex} and \eqref{eq:MoreParticularGreen}, we get
\begin{equation*}
\begin{split}
\alpha^2 & \int\limits_a^b\int\limits_c^d\left(H\left\{\hat{u}\right\}
-H\left\{u\right\}\right)(b-x)^{\alpha-1}(d-y)^{\alpha-1}dydx\\
& \geq\alpha^2\int\limits_a^b\int\limits_c^d
\mu\partial _3 H\left\{u\right\}(b-x)^{\alpha-1}(d-y)^{\alpha-1}dydx\\
& \quad +\alpha^2\int\limits_a^b\int\limits_c^d\left(_{a}D_{x}^\alpha[1]
\mu\partial _4 H\left\{u\right\}+_{c}D_{y}^\alpha[2]\mu
\partial_5 H\left\{u\right\}\right)(b-x)^{\alpha-1}(d-y)^{\alpha-1}dydx\\
& =\alpha^2\int\limits_a^b\int\limits_c^d\mu
\partial_3 H\left\{u\right\}(b-x)^{\alpha-1}(d-y)^{\alpha-1}dydx\\
& \quad+\alpha^2\int\limits_a^b\int\limits_c^d\left(_{a}D_{x}^\alpha[1]
\partial_4 H\left\{u\right\}+_{c}D_{y}^\alpha[2]
\partial_5 H\left\{u\right\}\right)\mu(b-x)^{\alpha-1}(d-y)^{\alpha-1}dydx\\
& =\alpha^2\int\limits_a^b\int\limits_c^d\left(\partial_3 H\left\{u\right\}
+_{a}D_{x}^\alpha[1]\partial _4 H\left\{u\right\}\right.\\
& \quad \left.+_{c}D_{y}^\alpha[2]
\partial_5 H\left\{u\right\}\right)\mu(b-x)^{\alpha-1}(d-y)^{\alpha-1}dydx\\
& =0.
\end{split}
\end{equation*}
Thus, $J[\hat{u}(\cdot,\cdot)]\geq J[u(\cdot,\cdot)]$.
\end{proof}


\section{Conclusion}

The fractional calculus provides a very useful framework
to deal with nonlocal dynamics: if one wants to include memory effects,
\textrm{i.e.}, the influence of the past on the behaviour
of the system at present time, then one may use fractional derivatives.
The proof of fractional Euler--Lagrange equations
is a subject of strong current study
because of its numerous applications. However, while the single time
case is well developed, the multitime fractional variational theory
is in its childhood, and much remains to be done.
In this work we consider a new class of multitime fractional
functionals of the calculus of variations subject to isoperimetric
constraints. We prove both necessary and sufficient optimality conditions
via the modified Riemann--Liouville approach.


\section*{Acknowledgments}

This work is part of the first author's Ph.D. project carried out at
the University of Aveiro under the framework of the Doctoral
Programme \emph{Mathematics and Applications} of Universities
of Aveiro and Minho. The financial support of
\emph{The Portuguese Foundation for Science and Technology} (FCT),
through the Ph.D. fellowship SFRH/BD/33865/2009, is here gratefully
acknowledged. The authors were also supported by FCT through the
\emph{Center for Research and Development in Mathematics and Applications}
(CIDMA). The second author would like to express his gratitude
to Professor Udri\c ste, for all the hospitality during
\emph{The International Conference of Differential Geometry
and Dynamical Systems} (DGDS-2010), held 25-28 August 2010
at the University Politehnica of Bucharest,
and for several stimulating books.



\end{document}